\documentclass[11pt]{article}

\usepackage[latin1]{inputenc}
\usepackage{amsfonts,amsmath,amstext,amsbsy,amsthm,amscd,amssymb}
\usepackage{multirow}
\usepackage{graphicx}
\usepackage{epsf,epsfig}
\def\N{\mathbb{N}}
\def\R{\mathbb{R}}

\def\interior#1{\smash{\mathop{#1}\limits^{\lower1pt\hbox{$\scriptscriptstyle\circ$}}}}

\newtheorem{theo}{Theorem}[section]

\newtheorem{prop}{Proposition}[section]
\newtheorem{lemma}{Lemma}[section]

\newenvironment{rem}{\noindent {\it Remark} : }{$\Diamond$ }

\begin{document}
\begin{titlepage}
\title{{\bf  Filtering the Wright-Fisher diffusion.}} 
\author{MIREILLE CHALEYAT-MAUREL$^{1}$, \and VALENTINE GENON-CATALOT$^{2}$}
\date{}


\maketitle

\thispagestyle{empty}

{\sc \it \noindent
$^{1}$(Corresponding author) Laboratoire MAP5, Universit\'e Paris Descartes, U.F.R. de Math\'ematiques
et Informatique, CNRS UMR 8145 and Laboratoire de Probabilit\'es et Mod\`eles Al\'eatoires (CNRS-UMR 7599), \\
45, rue des Saints-P\`eres, 75270 Paris Cedex 06, France.\\ e-mail:
mcm@math-info.univ-paris5.fr.\\
$^{2}$Laboratoire MAP5, Universit\'e Paris Descartes, U.F.R. de Math\'ematiques
et Informatique, CNRS-UMR 8145,\\ 45, rue des Saints-P\`eres,
75270 Paris Cedex 06, France. \\e-mail: genon@math-info.univ-paris5.fr}

\begin{abstract}
We consider a Wright-Fisher diffusion $(x(t))$ whose current state cannot be observed directly. Instead, at times $t_1<t_2<\ldots$, the observations $y(t_i)$ are such that, given the process $(x(t))$, the random variables $(y(t_i))$ are independent and the conditional distribution of $y(t_i)$ only depends on $x(t_i)$. When this conditional distribution has a specific form, we prove that the model $((x(t_i),y(t_i)), i\ge 1)$ is a computable filter in the sense that all distributions involved in filtering, prediction and smoothing are exactly computable. These distributions are expressed as finite mixtures of parametric distributions. Thus, the number of statistics to compute at each iteration is finite, but this number may vary along iterations.

\end{abstract}

\vspace{1cm}

\vspace{1cm}

\vspace{0,5cm}
\noindent
{\bf Running title:} Wright-Fisher diffusion.
\end{titlepage}

\section{Introduction}
Consider a large population  composed of two types of individuals 
A and a. Suppose that the proportion $x(t)$ of A-type at
time $t$ evolves continuously according to the following stochastic
differential equation
\begin{equation} \label{wf}
dx(t)= [-\delta x(t)+ \delta' (1-x(t))] dt +2 [x(t)(1-x(t))]^{1/2} dW_t, \quad
x(0)= \eta,
\end{equation}
where $(W_t)$ is a standard one-dimensional Brownian motion and $\eta$
is a random variable with values in $(0,1)$ independent of
$(W_t)$. This process is known as the Wright-Fisher gene frequency
diffusion model with mutation effects. It has values in the interval $(0,1)$. It appears as the diffusion approximation of the discrete time
and space Wright-Fisher Markov chain and is used to model the
frequency of an allele A in a population of genes composed of two distinct
alleles A and a (see {\em e.g.} Karlin and Taylor (1981, p. 176-179
and 221-222) or Wai-Yuan (2002, Chap. 6)). Suppose now that the current state $x(t)$ cannot be
directly observed. Instead, at times $t_1, t_2, \ldots,t_n$ with $0
\le t_1<t_2 \ldots <t_n$, we have
observations $ y(t_i)$ such that, given the whole process $(x(t))$, the
random variables $y(t_i)$ are independent and the conditional
distribution of $y(t_i)$ only depends on the corresponding state
variable $x(t_i)$. More precisely, we consider the following discrete
conditional distributions. Either, a binomial distribution, {\em i.e.},
for  $N \ge 1$ an integer,
\begin{equation} \label{bin}
P(y(t_i)= y|x(t_i)=x)= \binom{N}{y} x^{y} (1-x)^{N-y}, \quad y=0,1,\ldots,N,
\end{equation}
or, a negative binomial distribution, {\em i.e.}, for $m \ge 1$ an integer,
\begin{equation} \label{binneg}
P(y(t_i)= y|x(t_i)=x)=\binom{m+y-1}{y} x^{m} (1-x)^{y}, \quad y=0,1,2, \ldots.
\end{equation}
Under these assumptions, the
joint process  $(x(t_n), y(t_n))$ is a hidden Markov model (see {\em
  e.g.} Capp\'e et {\em al.}, 2005).

  In this context, a  central problem  that has been the subject of a
 huge number of contributions is the problem of  filtering, prediction
 or smoothing, {\em
   i.e.} the study of the conditional distributions of $x(t_l)$ given
 $y(t_n), \ldots, y(t_{1})$, with $l=n$ (filtering), $l=n+1,n+2, \ldots$
 (prediction), $l<n$ (smoothing). These distributions are generally
 called filters (respectively exact, prediction or marginal smoothing
 filters).  Although they may  be calculated
 recursively by  explicit algorithms,  iterations become rapidly
 intractable and exact formulae are difficult to
 obtain. To overcome this difficulty, authors generally try to find a
 parametric family ${\cal F}$ of distributions on the state space of $(x(t_n))$ ({\em i.e.} a family of distributions specified
 by a finite fixed number of real parameters) such that if ${\cal L}(x_0) \in {\cal F}$, then,
 for all $n,l$, ${\cal L}(x(t_l)|y(t_n), \ldots, y(t_{1}))$  belongs to ${\cal F}$. For such models, the term finite-dimensional filters is usually employed. This situation
 is illustrated by the linear Gaussian Kalman filter (see {\em e.g.}
 Capp\'e {\em et al.} 2005). There are few models satisfying the same
 properties as the Kalman filter: It is rather restrictive to impose a
 parametric family with a fixed number of parameters (see  Sawitzki
 (1981); see also Runggaldier and Spizzichino (2001)). 

Recently, new models where explicit computations are possible and which 	are not finite-dimensional filters have been proposed (see Genon-Catalot (2003), and Genon-Catalot and Kessler (2004)). Moreover, in a previous paper (Chaleyat-Maurel and Genon-Catalot, 2006), we have
introduced the notion of computable filters for the problem of
filtering and prediction. Instead of considering a parametric
class ${\cal F}$, we consider an enlarged class built using mixtures
of parametric distributions. The conditional distributions are specified by a finite number of
parameters, but this number may vary according to $n,l$. Still,
filters are computable explicitly. We give sufficient conditions on the
transition operator of $(x(t))$ and on the conditional distribution of
$y(t_i)$ given $x(t_i)$ to obtain such kind of filters. In the present paper,
we show that these conditions are satisfied by the model
above. Therefore, the conditional distributions of filtering and prediction are computable and we give the exact algorithm leading to these distributions. Moreover, we obtain the marginal smoothing distributions which are also given by an explicit and exact algorithm.  

The paper is organized as follows. In Section  \ref{pwf}, we briefly
recall some properties of the Wright-Fisher diffusion. In Section
\ref{suff}, we recall  the filtering-prediction algorithm and the
sufficient conditions of Chaleyat-Maurel and Genon-Catalot (2006) to
obtain computable filters. Section \ref{main} contains our main
results.  We introduce the class ${\bar {\cal F}}_{f}$ composed of
finite mixtures of parametric distributions fitted to the model (see
(\ref{fbar})).  We prove that the sufficient conditions hold for this
class and give the explicit formulae for the up-dating and the
prediction operator (Proposition \ref{c1}, Theorem \ref{miji},
Proposition \ref{mixiji}). The result concerning the prediction
operator is the most difficult part and requires several steps. Then,
we turn back in more details to the filtering-prediction algorithm
(Proposition \ref{algo}). Moreover, we give the exact distribution of
$(y(t_i), i=1, \ldots,n)$  which is also explicit. Hence, if
$\delta,\delta'$ are unknown and are to be estimated from the data set
$(y(t_i), i=1, \ldots,n)$, the exact maximum likelihood estimators of
these parameters can
be computed. At last, in Subsection \ref{margsmooth}, we study
marginal smoothing. We recall some classical formulae for computing
the marginal smoothing distributions. These formulae involve the
filtering distributions, that we have obtained in the previous
section, and  complementary terms that can be computed thanks to
Theorem \ref{miji}. In the Appendix, some technical proofs and auxiliary results are gathered.

\section{Properties of the Wright-Fisher diffusion model.}\label{pwf}
In order to exhibit the adequate class of distributions within which
the filters evolve, we need to recall some elementary properties of
model (\ref{wf}). The scale density is given by:
\begin{equation} \nonumber
s(x)= \exp{(-(1/2) \int^{x} \frac{-\delta u + \delta'(1-u)}{u(1-u)} du)}= x^{-\delta'/2} (1-x)^{-\delta/2}, \quad x \in (0,1).
\end{equation}
It satifies $\int_0 s(x) dx= \infty= \int^{1} s(x) dx$ if and only if
$\delta \ge 2$ and $\delta' \ge 2$, conditions that we assume from now
on. The speed density is equal to $m(x)= x^{\delta'/2 -1}
(1-x)^{\delta/2 -1}$, $x \in (0,1)$. Therefore, the unique stationary
distribution of (\ref{wf}) is the Beta distribution with parameters
$\delta'/2, \delta/2$  which has density
\begin{equation} \label{beta}
\pi(x)= \frac{x^{\frac{\delta'}{2} -1}
(1-x)^{\frac{\delta}{2} -1}}{B(\frac{\delta'}{2}, \frac{\delta}{2})} 1_{(0,1)}(x).
\end{equation}
For simplicity, we assume that the instants of observations are
equally spaced with sampling interval $\Delta$, {\em i.e.} $t_n= n
\Delta$, $n \ge 1$. Hence, the process $(X_n:=x(t_n))$ is a time-homogeneous
Markov chain. We denote by $p_{\Delta}(x,x')$ its transition density
and by $P_{\Delta}$ its transition operator. The transition density is
not explicitly known. However, it has a precise spectral expansion
(see {\em e.g.} Karlin and Taylor, 1981, p.335-336: Note that
$2x(t)-1$ is a Jacobi diffusion process). The results we obtain below
are linked with this spectral expansion although we do not use it directly (see the Appendix).
\section{Sufficient conditions for computable 
  filters.} \label{suff}
First, we focus on filtering and prediction and we consider case
(\ref{bin}) with $N=1$ for the conditional distributions of $Y_i:=y(t_i)$ given
$X_i=x(t_i)$, {\em i.e.}, we consider a Bernoulli conditional distribution. The other cases can be easily deduced afterwards (see the Appendix). Let us set
\begin{equation} \label{fxy}
P(Y_i=y|X_i=x)=f_x(y)= x^{y} (1-x)^{1-y}, y=0,1, x \in (0,1).
\end{equation}
We consider, on the finite set $\{0,1\}$, the dominating measure
$\mu(y)=1, y=0,1$. Thus, $f_x(y)$ is the density of ${\cal
  L}(Y_i|X_i=x)$ with respect to $\mu$.
\subsection{Conditional distributions for filtering, prediction and
  statistical inference.}
Denote by 
\begin{equation} \label{cond}
\nu_{l|n:1}={\cal L}(X_{l}\mid Y_{n},\ldots,Y_1),
\end{equation}
the conditional distribution of $X_l$ given
$(Y_n, Y_{n-1}, \ldots, Y_1)$. For $l=n$, this distribution is called
the optimal or exact filter and is used to estimate the unobserved
variable $X_n$ in an {\em on-line} way.  For $l=n+1$, the distribution
is called the prediction filter and is used to predict $X_{n+1}$ from
past values of the $Y_i$'s. For $1\le l<n$, it is a marginal smoothing distribution and is
used to estimate past data or to improve estimates obtained by exact filters.
 
It is well known that the exact and prediction filters can be obtained
recursively (see {\em e.g.} Capp\'e {\em et al.} (2005)). First,
starting with $\nu_{1|0:1}={\cal L}(X_1)$, we have
\begin{equation} \label{exactfilter}
\nu_{n|n:1}(dx) \propto \nu_{n|n-1:1}(dx) f_x(Y_n).
\end{equation}
Hence,
\begin{equation} \label{updating}
\nu_{n|n:1}= \varphi_{Y_n}( \nu_{n|n-1:1})
\end{equation}
is obtained by the operator $\varphi_y$ with $y=Y_n$ where, for $\nu$
a probability on $(0,1)$, 
$\varphi_{y}(\nu)$ is defined by:
\begin{equation} \label{phi}
\varphi_{y}(\nu)(dx)= \frac{f_{x}(y) \nu(dx)}{p_{\nu}(y)}, \quad
\mbox{with} \quad p_{\nu}(y)=  \int_{(0,1)}\nu(d\xi) f_{\xi}(y).
\end{equation}
 This step is the
up-dating step which allows to take into account a new
observation. Then, we have the prediction step
\begin{equation} \label{prediction}
\nu_{n+1|n:1}(dx')= \int_{(0,1)} \nu_{n|n:1}(dx) p_{\Delta}(x,x') dx'=\nu_{n|n:1}P_{\Delta}(dx'),
\end{equation} 
which consists in applying the transition operator: $\nu \rightarrow
\nu P_{\Delta}$.
These properties are obtained using that the joint process
$(X_n,Y_n)$ is Markov with transition
$p_{\Delta}(x_n,x_{n+1})f_{x_{n+1}}(y_{n+1}) dx_{n+1} \mu(dy_{n+1})$, and
initial distribution $\nu_{1|0:1}(dx_1) f_{x_1}(y_1)\mu(dy_1)$. Moreover,
the conditional distribution of $Y_n$ given $(Y_{n-1},\ldots, Y_1)$
has a density with respect to $\mu$, given by
\begin{equation} \label{marginalen}
p_{n|n-1:1}(y_n)= p_{\nu_{n|n-1:1}}(y_n)= \int_{(0,1)}
\nu_{n|n-1:1}(dx) f_x(y_n).
\end{equation}
Note that $(Y_n)$ is not Markov and  that the above
distribution effectively depends on all previous variables. For statistical inference based on $(Y_1, \ldots, Y_n)$,
the exact likelihood is given by
\begin{equation} \label{likelihood}
p_n(Y_1, \ldots, Y_n)= \prod_{i=1}^{n} p_{\nu_{i|i-1:1}}(Y_i).
\end{equation}

\subsection{Sufficient conditions for computable filters.}
Now, we recall the sufficient conditions of Chaleyat-Maurel and
Genon-Catalot (2006). First, consider a parametric class ${\cal F}=
\{\nu_{\theta}, \theta \in \Theta\}$ of
distributions on $(0,1)$, where $\Theta$ is a parameter set included
in $\R^{p}$, such that:
\begin{itemize}
\item [$\bullet$ C1:] For $y=0,1$, for all $\nu \in {\cal F}$,
  $\varphi_{y}(\nu)$ belongs to ${\cal F}$, {\em i.e.}
  $\varphi_{y}(\nu_\theta)= \nu_{T_{y}(\theta)}$, for some
  $T_{y}(\theta) \in \Theta$.
\item [$\bullet$ C2:] For all $\nu \in {\cal F}$, $\nu P_{\Delta}= \sum_{\lambda \in \Lambda}
  \alpha_\lambda \nu_{\theta_\lambda}$ is a finite mixture of elements
  of the class ${\cal F}$, {\em i.e.} $\Lambda$ is a finite set,
  $\alpha= (\alpha_\lambda, \lambda \in \Lambda)$ is a mixture
  parameter such that, for all $\lambda$, $\alpha_\lambda \ge 0$ and
  $\sum_{\lambda \in \Lambda} \alpha_\lambda =1$, and $\theta_\lambda
  \in \Theta$, for all $\lambda \in \Lambda$.
\end{itemize}
\begin{prop}\label{suffi}
Consider now the extended class ${\bar {\cal F}}_{f}$ composed of finite
mixtures of distributions of ${\cal F}$. Then, under (C1)-(C2), the 
operators $\varphi_y$, $y=0,1$ and $\nu \rightarrow \nu P_{\Delta}$
are from  ${\bar {\cal F}}_{f}$ into  ${\bar {\cal
    F}}_{f}$. Therefore, once starting with $\nu_{1|0:1}={\cal
  L}(X_1) \in  {\bar {\cal F}}_{f}$, all the distributions
$\nu_{n|n:1}$  (exact filters) and $\nu_{n+1|n:1}$ (prediction
filters) belong to  ${\bar {\cal F}}_{f}$. 
\end{prop}
For all $n$, these
distributions are completely specified by their mixture parameter and
the finite set of distributions involved in the mixture. Of course,
the number of components may vary along the iterations, but still
remains finite. Thus, these distributions are explicit and we say that filters are computable. 

The proof of Proposition \ref{suffi} is elementary (see Theorem 2.1, p.1451,
Chaleyat-Maurel and Genon-Catalot, 2006, see also Proposition \ref{algo} below). Now, it has a true impact
because the extended class is considerably larger than the initial
parametric class. Evidently, the difficulty is to find models
satisfying these conditions. Examples are given in Chaleyat-Maurel and
Genon-Catalot (2006): The hidden Markov process $(x(t))$ is a radial
Ornstein-Uhlenbeck process and two cases of conditional distributions
of $y(t_i)$ given $x(t_i)$ are proposed. Other examples are given in Genon-Catalot (2003) and Genon-Catalot and Kessler (2004).
Here, we study a new and completely different model. It has the
noteworthy feature that, to compute filters, the explicit formula for the
transition density of $(x(t))$ is not required, contrary to the
examples of the previous papers. (Details on the transition density
are given in the Appendix).

\section{Main results.}\label{main}
Our main results consists in exhibiting the proper parametric class of
distributions on $(0,1)$ and in checking (C1)-(C2) for this class and
the model specified by (\ref{wf}) and (\ref{fxy}). The interest of these
conditions is that they can be checked separately. Condition (C1) only
concerns the conditional distributions of $Y_i$ given $X_i$ and the
class ${\cal F}$.  In the Appendix, we prove condition (C1) for the
model specified by (\ref{wf}) and (\ref{bin}) or (\ref{binneg}). Condition
(C2) concerns the transition operator of $(x(t))$. It is the
most difficult part. Let us define the following class of distributions indexed
by $\Theta= \N \times \N$:
\begin{equation} \label{f}
{\cal F}= \{ \nu_{i,j}(dx)  \propto h_{i,j}(x) \pi(x) dx, (i,j) \in \N \times \N\},
\end{equation}
where
\begin{equation} \label{hij}
h_{i,j}(x)= x^{i} (1-x)^{j}.
\end{equation}
Hence, each distribution in ${\cal F}$ is a Beta distribution with
parameters $(i+ \frac{\delta'}{2},j+ \frac{\delta}{2})$ and (see (\ref{beta}))
\begin{equation} \label{nuij}
 \nu_{i,j}(dx)= \frac{x^{i+\frac{\delta'}{2} -1}
(1-x)^{j+\frac{\delta}{2} -1}}{B(i+\frac{\delta'}{2},
j+\frac{\delta}{2})} 1_{(0,1)}(x) dx.
\end{equation}
Let us define the extended class:
\begin{equation} \label{fbar}
 {\bar {\cal F}}_{f}= \{\sum_{(i,j) \in \Lambda}
  \alpha_{i,j} \nu_{i,j}, \Lambda \subset \N \times \N,
  |\Lambda|<+\infty, \alpha=(\alpha_{i,j}, (i,j) \in \Lambda) \in S_f\},
\end{equation}
where 
\begin{equation} \label{sf}
S_f= \{\alpha=(\alpha_{i,j}, (i,j) \in \Lambda), \Lambda \subset \N \times \N,
  |\Lambda|<+\infty, \forall (i,j), \alpha_{i,j} \ge 0,
  \sum_{(i,j) \in \Lambda}\alpha_{i,j}=1 \}
\end{equation}
is the set of finite mixture parameters. It is worth noting that the
stationary distribution $\pi(x)dx= \nu_{0,0}(dx)$ belongs to ${\cal
  F}$. Thus, in the important case where the initial distribution, {\em i.e.} the
distribution of $\eta$ (see (\ref{wf})), is the stationary
distribution, the exact and optimal filters have an explicit formula.

\subsection{Condition (C1): Conjugacy.}
\begin{prop} \label{c1} Let $\nu_{i,j} \in {\cal F}$ (see (\ref{f})). 
\begin{enumerate}
\item  For $y=0,1$, $\varphi_y(\nu_{i,j})= \nu_{i+y, j+1-y}.$ Hence, (C1) holds.
\item The marginal distribution
  is given by
\begin{equation} \label{margij}
p_{\nu_{i,j}}(y)= \left(\frac{i +\frac{\delta'}{2}}{i+j+ \frac{\delta'+\delta}{2}}\right)^{y}\left(\frac{j +\frac{\delta}{2}}{i+j+ \frac{\delta'+\delta}{2}}\right)^{1-y}, y=0,1.
\end{equation}
\end{enumerate}
\end{prop}
\begin{proof} The first point is obtained using that $f_x(y)
  \nu_{i,j}(dx)\propto x^{i+y +\frac{\delta'}{2}-1}(1-x)^{j+1-y+
    \frac{\delta}{2}-1} dx\propto h_{i+y,j+1-y}(x) \pi(x)dx$. 

For the marginal
  distribution, we have 
\begin{equation}\nonumber
p_{\nu_{i,j}}(y)= \frac{B(i+y+ \frac{\delta'}{2}, j+1-y+ \frac{\delta}{2})}{B(i +\frac{\delta'}{2},j +\frac{\delta}{2})}.
\end{equation}
To get (\ref{margij}), we use the classical relations $B(a,b)=\frac{\Gamma(a)\Gamma(b)}{\Gamma(a+b)}$ and 
$\Gamma(a+1)=a\Gamma(a)$, $a,b>0$ where $\Gamma(.)$ is the standard Gamma function.
\end{proof}
Proposition \ref{c1} is proved for the conditional distributions (\ref{bin}) and (\ref{binneg}) in the Appendix.

\begin{rem}\label{conjug}
Proposition \ref{c1} states that the class ${\cal F}$ is a conjugate
class for the parametric family of distributions on $\{0,1\}$: $y
\rightarrow f_x(y)\mu(dy)$ with respect to the parameter $x \in
(0,1)$. The property that Beta distributions are conjugate with
respect to Bernoulli distributions is well known in Bayesian statistics (see {\em
  e.g.} West and Harrison (1997)). Analogously, the class ${\cal F}$ is also a conjugate
class for the parametric family of distributions (\ref{bin}) and
(\ref{binneg}) with respect to $x \in
(0,1)$.
\end{rem}
\subsection{Condition (C2): Introducing mixtures.}
The class ${\cal F}$ is the natural class of distributions to consider for the conditional distributions (\ref{fxy}) because this class contains the stationary distribution  $\pi(x)dx=\nu_{0,0}(dx)$ of (\ref{wf}). We do have $\nu_{0,0} P_{\Delta}=\nu_{0,0}$. However, when $(i,j) \neq (0,0)$, $\nu_{i,j} P_{\Delta}$  no more belongs to ${\cal F}$ but belongs to ${\bar {\cal F}}_{f}$ as we prove below. We need some preliminary properties.
\begin{prop}\label{revers}
For all Borel positive function $h$ defined on $(0,1)$,  if $\nu(dx)= h(x) \pi(x) dx$, then, $\nu P_{t}(dx')= P_{t}h(x') \pi(x') dx'$, where, for $t \ge 0$,
$$
P_{t}h(x')= \int_0^{1}h(x)p_{t}(x',x)dx.
$$
\end{prop}
\begin{proof}It is well known that one-dimensional diffusion processes are reversible with respect to their speed density: The transition $p_t(x,x')$ is reversible with respect to $\pi(x)dx$, {\em i.e.} satisfies for all $(x,x') \in (0,1) \times (0,1)$:
\begin{equation}\label{rev}
\pi(x) p_t(x,x')= \pi(x') p_t(x',x).
\end{equation}
This gives the result.
\end{proof}
\begin{prop}\label{pth} Suppose that:
\begin{itemize}
\item [$\bullet$] (C3) For all $(n,p) \in \N \times \N$, there exists
  a set $\Lambda_{n,p} \subset \N \times \N$ such that $|\Lambda_{n,p}|<+\infty$  and for all $t\ge 0$, 
\begin{equation}\label{pthnp}
P_t h_{n,p}(.)= \sum_{(i,j) \in \Lambda_{n,p}}B_{i,j}(t) h_{i,j}(.), 
\end{equation}
with, for all $(i,j)$, and all $t\ge 0$, $B_{i,j}(t) \ge 0$. 
\end{itemize}
Then, condition (C2) holds for $P_t$ for all $t \ge 0$. Moreover, 
\begin{equation}\nonumber
\nu_{n,p}P_t(dx) = \sum_{(i,j) \in \Lambda_{n,p}}\alpha_{i,j}(t) \nu_{i,j}(dx), 
\end{equation}
where $\alpha(t)=(\alpha_{i,j}(t), (i,j) \in \Lambda_{n,p})$ belongs to $S_f$ (see (\ref{sf})) and  
\begin{equation}\label{alphaij}
\alpha_{i,j}(t)= B_{i,j}(t)  \frac{B(i+\frac{\delta'}{2},j+\frac{\delta}{2})}{B(n+\frac{\delta'}{2},p+\frac{\delta}{2})}.
\end{equation} 
\end{prop}
\begin{proof}We have (see (\ref{nuij})), for all$(n,p)$,
\begin{equation}\nonumber
\nu_{n,p}P_t =\frac{B(\frac{\delta'}{2},\frac{\delta}{2})}{B(n+\frac{\delta'}{2},p+\frac{\delta}{2})}(h_{n,p} \pi)P_t.
\end{equation}
Using Proposition \ref{revers}, we get:
\begin{equation}\nonumber
(h_{n,p}\pi)P_t (dx)=P_t h_{n,p}(x) \pi(x) dx.
\end{equation}
Now,
\begin{equation}\nonumber
P_t h_{n,p}(.) \pi(.) = \sum_{(i,j) \in \Lambda_{n,p}}B_{i,j}(t) h_{i,j}(.) \pi(.),
\end{equation}
and
\begin{equation}\nonumber
h_{i,j}(x) \pi(x) dx = \nu_{i,j}(dx)\frac{B(i+\delta'/2,j+\delta/2)}{B(\delta'/2, \delta/2)}.
\end{equation}
Joining all formulae, we get the result.
\end{proof}
Therefore, it remains to prove condition (C3). We start with a classical lemma.
\begin{lemma}\label{L} Let $h$ belong to the set $C_{b}^{2}((0,1))$ of bounded and twice continuously differentiable functions on $(0,1)$. Then, 
\begin{equation}\label{dt}
\frac{d}{dt}(P_th(x))= P_t L h(x), P_0h(x)=h(x),
\end{equation}
where $Lh(x)= 2 x(1-x) h''(x)+ [-\delta x + \delta'(1-x)]h'(x)$ is the infinitesimal generator of (\ref{wf}).
\end{lemma}
\begin{proof}Let $x_x(t)$ be the solution of (\ref{wf}) with initial condition $x_x(0)=x$. By the Ito formula,
$$
h(x_x(t))=h(x) + \int_0^{t}Lh(x_x(s)) ds +2 \int_0^{t}h'(x_x(s)) \left(x_x(s)(1-x_x(s))\right)^{1/2}dW_s.
$$
By the assumption on $h$, taking expectations yields:
$$
P_th(x)= h(x) + \int_0^{t} P_s Lh(x) ds,
$$
which is the result.
\end{proof}
Now, we start to compute $P_t h_{n,p}(x)$ for all $n,p \in \N$.
\begin{prop}\label{mnp} Let $m_{n,p}(t,.)= P_t h_{n,p}(.)$ Then, for all $n,p \in \N$,
\begin{equation} \label{dmnp}
\frac{d}{dt}m_{n,p}(t,.)= -a_{n+p}m_{n,p}(t,.) + c_n(\delta') m_{n-1,p}(t,.) + c_p(\delta) m_{n,p-1}(t,.), \quad m_{0,0}(t,.)=1,
\end{equation}
where, for all $n \in \N$,
\begin{equation}\label{ancn}
a_n=n[2(n-1)+\delta+\delta'],\quad  c_n(\delta)=n[2(n-1)+\delta].
\end{equation}
(If $p$ or $n$ is equal to $0$, then, $c_p(\delta)=0$
or $c_n(\delta')=0$, and  formula (\ref{dmnp}) holds). Note that,
since the expression of $a_n$ is symetric with respect to
$\delta,\delta'$, we do not mention the dependance on these parameters. Note also that for all $n$, since both $\delta$ and $\delta'$ are positive (actually $\ge 2$), the coefficients $a_n$ and $c_n(\delta),c_n(\delta')$ are non negative.
\end{prop}
\begin{proof} In view of (\ref{dt}), it is enough to prove that
\begin{equation} \label{lhnp}
Lh_{n,p} = - a_{n+p} h_{n,p} + c_n(\delta') h_{n-1,p} + c_p(\delta) h_{n,p-1},
\end{equation}
where $L$ is defined in Lemma \ref{L}. To make the proof clear, let us
start with computing $Lh_{n,0}$. We have immediately:
$$
Lh_{n,0}(x)= - a_n h_{n,0}(x) + c_n(\delta') h_{n-1,0}(x).
$$
Now, since $y(t)=1-x(t)$ satifies 
\begin{equation}\label{yt}
dy(t)= [-\delta' y(t) +\delta(1-y(t))]dt + 2 (y(t)(1-y(t))^{1/2} dW_t,
\end{equation}
we obtain $Lh_{0,p}$ by simply interchanging $\delta$ and $\delta'$
and get
$$
Lh_{0,p}(x)= - a_p h_{0,p}(x) + c_p(\delta) h_{0,p-1}(x).
$$
Finally, to compute $Lh_{n,p}$, we use the
following tricks: Each time
$x^{n+1}$ appears, we write $x^{n+1}= -(1-x-1)x^{n}= -(1-x)x^{n} +
x^{n}$; each time $(1-x)^{p+1}$ appears, we write
$(1-x)^{p+1}=(1-x)^{p} (1-x)=(1-x)^{p}- x(1-x)^{p}$. Grouping terms,
we get (\ref{lhnp}). 
\end{proof}
Our aim is now to prove that
\begin{equation}\label{final}
m_{n,p}(t,.)= \exp{(-a_{n+p} t)} h_{n,p}(.) + \sum_{0 \le k \le n, 0 \le l \le p,(k,l)\neq(0,0)}B_{n-k,p-l}^{n,p}(t) h_{n-k,p-l}(.),
\end{equation}
where, for all $(k,l),(n,p)$, $B_{n-k,p-l}^{n,p}(t) \ge 0$ for all $t \ge 0$.  Moreover, we give below the precise formula for these coefficients. Hence, the set $\Lambda_{n,p}$ of (C3) is equal to $\{(k,l), 0 \le k \le n, 0 \le l \le p\}$.  The first term can also be denoted by
$$
B_{n,p}^{n,p}(t)= \exp{(-a_{n+p} t)}.
$$
It has a special role because it is immediately obtained by (\ref{dmnp}).

\subsubsection{Computation of $m_{n,0}(t,.)$ and $m_{0,n}(t,.)$.}
Recall notation (\ref{hij}) and that $m_{n,0}(t,.)=P_t h_{n,0}(.)$. We prove now that (\ref{final}) holds for all $(n,0)$ and all $(0,n)$.
\begin{theo}\label{mn} The following holds:
\begin{equation}\label{mn0}
m_{n,0}(t,.)=\exp{(-a_n t)} h_{n,0}(.) + \sum_{k=1}^{n}B_{n-k,0}^{n,0}(t) h_{n-k,0}(.),
\end{equation}
where, for all $(k,n)$, with $1 \le k \le n$, $B_{n-k,0}^{n,0}(t) \ge 0$ for all $t \ge 0$. Moreover, for $k=1, \ldots,n$,
\begin{equation}\label{bkn}
B_{n-k,0}^{n,0}(t)=c_n(\delta') c_{n-1}(\delta')\ldots c_{n-k+1}(\delta') B_t(a_n, a_{n-1}, \ldots, a_{n-k}),
\end{equation}
where
\begin{equation}\label{btank}
B_t(a_n, a_{n-1}, \ldots, a_{n-k})=(-1)^{k}\sum_{j=0}^{k} \exp{(-a_{n-j} t)} \frac{(-1)^{j}}{\prod_{0 \le l \le k, l \neq j}|a_{n-j}-a_{n-l}|}.
\end{equation}
We can also set 
\begin{equation}\nonumber
B_{n,0}^{n,0}(t)= \exp{(-a_n t)}.
\end{equation}
Analogously:
\begin{equation}\label{m0n}
m_{0,n}(t,.)=\exp{(-a_n t)} h_{0,n}(.) + \sum_{k=1}^{n}B_{0,n-k}^{0,n}(t) h_{0,n-k}(.),
\end{equation}
where, for all $(k,n)$, with $1 \le k \le n$, $B_{0,n-k}^{0,n}(t) \ge 0$ for all $t \ge 0$. Moreover, for $k=1, \ldots,n$,
\begin{equation}\nonumber
B_{0,n-k}^{0,n}(t)=c_n(\delta) c_{n-1}(\delta)\ldots c_{n-k+1}(\delta) B_t(a_n, a_{n-1}, \ldots, a_{n-k}),
\end{equation}
We also set
\begin{equation}\nonumber
B_{0,n}^{0,n}(t)= \exp{(-a_n t)}.
\end{equation}
\end{theo}
\begin{proof}For the proof, let us fix $x$ and set $m_{n,0}(t,x)=m_n(t)$. We also set $B_{n-k,0}^{n,0}(t)=B_{n-k}^{n}(t)$ during this proof. Solving $m{'}_n(t)=- a_n m_n(t) + c_n(\delta') m_{n-1}(t), m_n(0)=x^{n}=h_{n,0}(x)$ yields
\begin{equation}\label{eqn}
m_n(t)=\exp{(-a_n t)}\;x^{n} + \exp{(-a_n t)}\int_0^{t}\exp{(a_n s)}m_{n-1}(s)ds.
\end{equation}
Let us first prove by induction that
\begin{equation}\label{induc}
m_n(t)= \sum_{k=0}^{n}B_{n-k}^{n}(t) x^{n-k},
\end{equation}
where $B_{n-k}^{n}(t) \ge 0$ for all $t \ge 0$ and all $k=0, \ldots,n$ and $B_n^{n}(t)=\exp{(-a_n t)}$. For $n=0$, $m_0(t)=1$. For $n=1$, we solve (\ref{eqn}) and get
\begin{equation}\nonumber
m_1(t)=\exp{(-a_1 t)} x + c_1(\delta')\frac{(1-\exp{(-a_1 t)})}{a_1}.
\end{equation}
So, (\ref{induc}) holds for $n=1$ with
\begin{equation}\label{b0}
B_1^{1}(t)= \exp{(-a_1 t)} , \quad B_0^{1}(t)= c_1(\delta')\frac{(1-\exp{(-a_1 t)})}{a_1}\ge 0.
\end{equation}
Suppose (\ref{induc}) holds for $n-1$. We now apply (\ref{eqn}). Identifying the coefficients of $x^{n-k}$, $0 \le k \le n$, we get:
$$
B_n^{n}(t)=\exp{(-a_n t)},
$$
and for $k=0,1, \ldots n-1$,
\begin{equation}\label{inducB}
B_{n-(k+1)}^{n}(t)= c_n(\delta') \exp{(-a_n t)}\int_0^{t} \exp{(a_n s)} B_{n-1-k}^{n-1}(s) ds.
\end{equation}
Hence, (\ref{induc}) holds for all $n \ge 0$ with all coefficients non negative. 

Now, we prove (\ref{bkn})-(\ref{btank}) by induction using (\ref{inducB}). For $n=1$, we look at (\ref{b0}) and see that, since $a_0=0$,
\begin{equation} \nonumber
B_0^{1}(t) =c_1(\delta') (-1) [\frac{\exp{(-a_1 t)}}{a_1- a_0}+\frac{(-1) \exp{(-a_0 t)}}{|a_0- a_1|}]= c_1(\delta') B_t(a_1, a_0).
\end{equation}
Now, suppose we have formulae (\ref{bkn})-(\ref{btank}) for $n-1$ and $k=0,1, \ldots, n-1$. We know that $B_n^n(t)= \exp{(-a_n t)}$. Let us compute, for $k=0, 1, \ldots,n-1$, $B_{n-(k+1)}^{n}(t)$ using (\ref{inducB}). We have:
\begin{equation}\label{formule1}
B_{n-(k+1)}^{n}(t)= c_n(\delta')c_{n-1}(\delta') \ldots c_{n-k}(\delta') (-1)^{k} \times B,
\end{equation}
with
\begin{equation}\label{formule2}
B=\sum_{j=0}^{k} \exp{(-a_n t)}\int_0^{t} \exp{((a_n-a_{n-1-j}) s)}ds \frac{(-1)^{j}}{\prod_{0 \le l \le k, l \neq j}|a_{n-1-j}-a_{n-1-l}|}.
\end{equation}
Integrating, we get:
\begin{equation}\nonumber
B=\sum_{j=0}^{k}  \exp{(-a_{n-1-j}t)} \frac{(-1)^{j}}{(a_n - a_{n-1-j}) \prod_{0 \le l \le k, l \neq j}|a_{n-1-j}-a_{n-1-l}|} + (- \exp{(-a_n t)}) A,
\end{equation}
with 
\begin{equation}\label{formule4}
A= \sum_{j=0}^{k} \frac{(-1)^{j}}{(a_n - a_{n-1-j}) \prod_{0 \le l \le k, l \neq j}|a_{n-1-j}-a_{n-1-l}|}.
\end{equation}
Hence,
\begin{equation}\label{formule5}
B=\sum_{j'=1}^{k+1}  \exp{(-a_{n-j'}t)} \frac{(-1)^{j'-1}}{ \prod_{0 \le l' \le k+1, l' \neq j'}|a_{n-j'}-a_{n-l'}|} + (- \exp{(-a_n t)}) A.
\end{equation}
In view of
(\ref{bkn})-(\ref{btank})-(\ref{formule1})-(\ref{formule4})-(\ref{formule5}),
to complete the proof of (\ref{mn0}), it remains to show the following equality:
\begin{lemma} \label{lemA}
\begin{equation} \nonumber
 \sum_{j=0}^{k} \frac{(-1)^{j}}{(a_n - a_{n-1-j}) \prod_{0 \le l \le k, l \neq j}|a_{n-1-j}-a_{n-1-l}|}= \frac{1}{(a_n - a_{n-1})(a_n - a_{n-2})\ldots (a_{n}-a_{n-k-1})}.
\end{equation}
\end{lemma}
This lemma requires some algebra and its proof is postponed to the
Appendix. At last, to get (\ref{m0n}), we just interchange $\delta'$ and $\delta$ in all formulae because of (\ref{yt}).
\end{proof}
\subsubsection{Computation of $m_{n,p}(t,.)$.}
Recall that $h_{n,p}(x)= x^{n}(1-x)^{p}$ and $m_{n,p}(t,.)=P_t h_{n,p}(.)$. Now, we focus  on formula (\ref{dmnp}). It is easy to see that, since we have computed $m_{n,0}(t,.)$ for all $n$ and $m_{0,p}(t,.)$ for all $p$, then, we deduce from (\ref{dmnp}) $m_{n,p}(t,.)$ for all $(n,p)$. This is done as  follows. Suppose we have computed all terms $m_{i,j-i}(t,.)$ for $0 \le i \le j \le n$, then, we obtain all terms $m_{i,j-i}(t,.)$ for $0 \le i \le j \le n+1$. Indeed, the extra terms are:
\begin{itemize}
\item $m_{n+1,0}(t,.)$ that we know already,
\item $m_{i,n+1-i}(t,.)$ for $0<i<n+1$ that is calculated from
\begin{equation}\nonumber
\frac{d}{dt}m_{i,n+1-i}(t,.)= -a_{n+1}m_{i,n+1-i}(t,.) + c_i(\delta') m_{i-1,n+1-i}(t,.) + c_{n+1-i}(\delta) m_{i,n+1-i-1}(t,.), 
\end{equation}
\item at last, $m_{0,n+1}(t,.)$ that we know already.
\end{itemize}
This is exactly filling in a matrix composed of the terms $m_{n,p}(t,.)$. Having the first line $m_{0,n}(t,.)$ and the first column $m_{n,0}(t,.)$, we get each new term $m_{i,j}(t,.)$ from the one above ($m_{i-1,j}(t,.)$) and the one on the left ($m_{i,j-1}(t,.)$).

Now, we proceed to get formula (\ref{final}).
\begin{theo}\label{miji} For all $(i,j)$ such that $0 \le i \le j\le n$, 
\begin{equation}\label{final2}
m_{i,j-i}(t,.)= \exp{(-a_{j} t)} h_{i,j-i}(.) + \sum_{0 \le k \le i, 0 \le l \le j-i,(k,l)\neq(0,0)}B_{i-k,j-i-l}^{i,j-i}(t) h_{i-k,j-i-l}(.),
\end{equation}
with 
\begin{equation}\nonumber
B_{i-k,j-i-l}^{i,j-i}(t)= \binom{k+l}{k}c_i(\delta') \ldots c_{i-k+1}(\delta') c_{j-i}(\delta) \ldots c_{j-i-l+1}(\delta) B_t(a_j,a_{j-1}, \ldots, a_{j-(k+l)}),
\end{equation}
with the convention that, for $k=0$, there is no term in $c_.(\delta')$ and for $l=0$, there is no term in $c_.(\delta)$).
\end{theo}
\begin{proof} By (\ref{bkn})-(\ref{btank})-(\ref{inducB}), we have proved that
\begin{equation}\label{inducbtank}
B_t(a_{n+1}, a_n, \ldots, a_{n-k})= \exp{(-a_{n+1}t)} \int_0^{t} \exp{(a_{n+1}s)} B_s(a_n, a_{n-1}, \ldots, a_{n-k}) ds.
\end{equation}
Suppose (\ref{final2}) holds for $0 \le i \le j \le n$. Let us compute the extra terms $m_{i,n+1-i}(t,.)$ for $0<i<n+1$ using their differential equations. We have
\begin{equation}\label{etape0}
m_{i,n+1-i}(t,.)= \exp{(-a_{n+1}t)} h_{i,n+1-i}(.) + A_i(\delta') + B_i(\delta),
\end{equation}
with
\begin{equation}\label{etape1}
A_i(\delta')= \exp{(-a_{n+1}t)} \int_0^{t} \exp{(a_{n+1}s)}\;c_i(\delta') m_{i-1,n+1-i}(s,.)ds,
\end{equation}
\begin{equation}\label{etape2}
B_i(\delta)= \exp{(-a_{n+1}t)} \int_0^{t} \exp{(a_{n+1}s)}\;c_{n+1-i}(\delta) m_{i,n+1-i-1}(s,.)ds.
\end{equation}
We apply the induction formula and replace $m_{i-1,n+1-i}(s,.), m_{i,n+1-i-1}(s,.)$ by their  development. This yields:
\begin{eqnarray*}
m_{i-1,n+1-i}(s,.)&=& \exp{(-a_{n} s)} h_{i-1,n+1-i}(.) \\
&+& \sum_{0 \le k' \le i-1, 0 \le l' \le n+1-i,(k',l')\neq(0,0)}B_{i-1-k',n+1-i-l'}^{i-1,n+1-i}(s) h_{i-1-k',n+1-i-l'}(.),
\end{eqnarray*}
\begin{eqnarray*}
m_{i,n+1-i-1}(s,.)&=& \exp{(-a_{n} s)} h_{i,n+1-i-1}(.) \\
&+ &\sum_{0 \le k'' \le i, 0 \le l'' \le n+1-i-1,(k'',l'')\neq(0,0)}B_{i-k'',n+1-i-1-l''}^{i,n+1-i-1}(s) h_{i-k'',n+1-i-1-l''}(.).
\end{eqnarray*}
In (\ref{etape0})-(\ref{etape1})-(\ref{etape2}), the coefficient of  $h_{i-1,n+1-i}(.)$ obtained by the above relations only comes from:
\begin{equation}\nonumber
c_i(\delta') \exp{(-a_{n+1}t)} \int_0^{t} \exp{(a_{n+1}s)} \exp{(- a_n s)}ds.
\end{equation}
By (\ref{inducbtank}), this term is equal to:
\begin{equation}\label{final211}
c_i(\delta') B_t(a_{n+1}, a_n) = B_{i-1, n+1-i-0}^{i,n+1-i} = \binom{1}{1}c_i(\delta') B_t(a_{n+1},a_n).
\end{equation}
Analogously, the coefficient of $h_{i,n+1-i-1}(.)$ comes from:
\begin{equation}\nonumber
c_{n+1-i}(\delta) \exp{(-a_{n+1}t)} \int_0^{t} \exp{(a_{n+1}s)} \exp{(- a_n s)}ds.
\end{equation}
This term is equal to:
\begin{equation}\label{final212}
c_{n+1-i}(\delta) B_t(a_{n+1}, a_n) = B_{i, n+1-i-1}^{i,n+1-i} = \binom{1}{0}c_{n+1-i}(\delta) B_t(a_{n+1},a_n).
\end{equation}
Now, the coefficient of the current term $h_{i-k,n+1-i-l}(.)$ comes from the sum of the following two terms:
\begin{equation}\label{(i)}
b_{1}=c_i(\delta') \exp{(-a_{n+1}t)} \int_0^{t} \exp{(a_{n+1}s)} B_{i-1-(k-1),n+1-i-l}^{i-1,n+1-i}(s)ds,
\end{equation}
($i-1-k'=i-k, n+1-i-l'=n+1-i-l$, thus $k'=k-1, l'=l$)
and
\begin{equation}\label{(ii)}
b_{2}=c_{n+1-i}(\delta) \exp{(-a_{n+1}t)} \int_0^{t} \exp{(a_{n+1}s)} B_{i-k,n+1-i-l}^{i,n+1-i-1}(s)ds, 
\end{equation}
($i-k''=i-k, n+1-i-1-l''=n+1-i-l$, thus $k''=k, l''=l-1$). Thus,
\begin{equation}\nonumber
\hspace{-2cm} b_{1}= c_i(\delta') \exp{(-a_{n+1}t)} \int_0^{t} \exp{(a_{n+1}s)} 
 \binom{k+l-1}{k-1}c_{i-1}(\delta') \ldots c_{i-1-(k-1)+1}(\delta') \times
\end{equation}
\begin{equation}\nonumber
 c_{n+1-i}(\delta)\ldots c_{n+1-i-l+1}(\delta) 
  B_s(a_n,a_{n-1},\ldots, a_{n-(k+l-1)}) ds
\end{equation}
\begin{equation}\nonumber
\hspace{-4cm} =\binom{k+l-1}{k-1}c_i(\delta') c_{i-1}(\delta') \ldots c_{i-k+1}(\delta') c_{n+1-i}(\delta)\ldots c_{n+1-i-l+1}(\delta) 
\end{equation}
\begin{equation}\nonumber
\times B_t(a_{n+1},a_n,a_{n-1},\ldots, a_{n+1-(k+l)}) .
\end{equation}
And 
\begin{equation}\nonumber
\hspace{-2cm} b_{2}= c_{n+1-i}(\delta) \exp{(-a_{n+1}t)} \int_0^{t} \exp{(a_{n+1}s)} 
 \binom{k+l-1}{k}c_{i}(\delta') \ldots c_{i-k+1}(\delta') \times
\end{equation}
\begin{equation}\nonumber
 c_{n+1-(i+1)}(\delta)\ldots c_{n+1-i-l+1}(\delta) 
  B_s(a_n,a_{n-1},\ldots, a_{n-(k+l-1)}) ds
\end{equation}
\begin{equation}\nonumber
\hspace{-4cm} =\binom{k+l-1}{k}c_i(\delta') c_{i-1}(\delta') \ldots c_{i-k+1}(\delta') c_{n+1-i}(\delta)\ldots c_{n+1-i-l+1}(\delta) 
\end{equation}
\begin{equation}\nonumber
\times B_t(a_{n+1},a_n,a_{n-1},\ldots, a_{n+1-(k+l)}) .
\end{equation}
Now using that $\binom{k+l-1}{k-1}+\binom{k+l-1}{k}=\binom{k+l}{k}$, we finally obtain that  $b_1 + b_2 $ is exactly equal to the expected term $B_{i-k,n+1-i-l}^{i,n+1-i}(t)$.
\end{proof}

\subsubsection{Back to the mixture coefficients.}
Now, we use Proposition \ref{pth} and formula (\ref{alphaij})  to obtain the mixture coefficients for $\nu_{i,j-i}P_t$ We begin with a lemma.
\begin{lemma}\label{b} For $0 \le i \le n$ and $0 \le j \le p$, we
  have
\begin{equation}\label{bb}
\frac{B(i+\frac{\delta'}{2},j+\frac{\delta}{2})}{B(n+\frac{\delta'}{2},p+\frac{\delta}{2})}= \frac{\binom{i+j}{j}a_{n+p}a_{n+p-1} \ldots a_{i+1+p}a_{i+p}a_{i+p-1}\ldots  a_{i+j+1}}{\binom{n+p}{p}c_n(\delta') c_{n-1}(\delta')\ldots c_{i+1}(\delta')c_p(\delta)c_{p-1}(\delta)\ldots c_{j+1}(\delta)}.
\end{equation}
In the trivial case $(i,j)=(n,p)$, the quotient is equal to $1$. In (\ref{bb}), for
$j=p$, there is no term in  $c_.(\delta)$, and for $i=n$, there is no term in $c_.(\delta')$.
\end{lemma}
\begin{proof} We use the relation
\begin{equation}
B(a+1, b+1)= \frac{a b}{(a+b+1)(a+b)}B(a,b). 
\end{equation} 
Hence, 
\begin{equation}\nonumber
B(n+\frac{\delta'}{2},p+\frac{\delta}{2})= \frac{(n-1+\frac{\delta'}{2})(p-1+\frac{\delta}{2})}{(n+p-1 +
  \frac{\delta'+\delta}{2})(n+p-2 +
  \frac{\delta'+\delta}{2})}B(n-1+\frac{\delta'}{2},p-1+\frac{\delta}{2}).
\end{equation}
By (\ref{ancn}),
we have:
\begin{equation}\nonumber
\frac{a_n}{2n}= n-1 +  \frac{\delta'+\delta}{2},\quad
\frac{c_n(\delta')}{2n}= n-1 +  \frac{\delta'}{2},\quad \frac{c_p(\delta)}{2p}= p-1 +  \frac{\delta}{2}.
\end{equation}
Hence, 
\begin{equation} \nonumber
B(n+\frac{\delta'}{2},p+\frac{\delta}{2})= \frac{c_n(\delta')}{2n}
\frac{c_p(\delta)}{2p}\frac{2(n+p)}{a_{n+p}}
\frac{2(n+p-1)}{a_{n+p-1}}B(n-1+\frac{\delta'}{2},p-1+\frac{\delta}{2}).
\end{equation}
Iterating downwards yields:
\begin{equation}\nonumber
\frac{B(i+\frac{\delta'}{2},j+\frac{\delta}{2})}{B(n+\frac{\delta'}{2},p+\frac{\delta}{2})}=\frac{(i+j)!}{(n+p)!}\frac{n!}{i!}\frac{p!}{j!}\frac{a_{n+j}a_{n+j-1}\ldots a_{i+j+1}}{c_n(\delta')c_{n-1}(\delta')\ldots c_{i+1}(\delta')c_p(\delta)c_{p-1}(\delta)\ldots c_{j+1}(\delta)}.
\end{equation}
This gives (\ref{bb}).
\end{proof}
Now, we have the complete formula for $\nu_{i,j-i}P_t$.
\begin{prop}\label{mixiji}
For $0 \le i \le j$, we have 
\begin{equation}\nonumber
\nu_{i,j-i}P_t= \sum_{k=0,\ldots ,i, l=0,\ldots , j-i} \alpha^{i,j-i}_{i-k,j-i-l}(t) \;\nu_{i-k,j-i-l},
\end{equation}
where, for $(k,l) \neq (0,0)$,
\begin{equation}\label{alphaiji}
\alpha^{i,j-i}_{i-k,j-i-l}(t) = \frac{\binom{i}{k}\binom{j-i}{l}}{\binom{j}{k+l}} a_j a_{j-1}\ldots a_{j-k-l+1} B_t(a_j,a_{j-1}, \ldots, a_{j-k-l}).
\end{equation}
For $(k,l)=(0,0)$,
\begin{equation}\label{alpha}
\alpha^{i,j-i}_{i,j-i}(t) = \exp{(-a_j t)}.
\end{equation}
\end{prop}
The proof is straightforward using Proposition \ref{pth}, Theorem \ref{miji} and the lemma. 

Let us now make some remarks concerning the above result. First, note that the mixture coefficients are symetric
  with respect to $\delta'$ and $\delta$. The non symetric part
  appears in the distributions $\nu_{i-k,j-i-l}$. Another point is that, looking at $\nu_{i,j-i}P_t$, we see that very few mixture coefficients will be significantly non nul. Indeed, they are all composed of sums of rapidly decaying exponentials.  

 To illustrate our result, let us compute more precisely some terms, {\em e.g.} $\nu_{1,0}, \nu_{2,0}$. For $n=1$, $\alpha_{1,0}^{1,0}(t)= \exp{(-a_1 t)}$,  $\alpha_{0,0}^{1,0}(t)=a_1 B_t(a_1,a_0)= 1-\exp{(-a_1 t)}$ and $a_1=\delta'+\delta, a_0=0$. Hence:
\begin{equation}\nonumber
\nu_{1,0}= \exp{(-(\delta'+\delta) t)} \nu_{1,0} + (1-\exp{(-(\delta'+\delta) t)}) \nu_{0,0}.
\end{equation}
For $n=2$, 
\begin{eqnarray*}
\alpha_{2,0}^{2,0}(t)&=& \exp{(-a_2 t)},\\
 \alpha_{1,0}^{2,0}(t)&=& a_2 B_t(a_2,a_1)= \frac{a_2}{a_2-a_1}(\exp{(-a_1 t)}-\exp{(-a_2 t)}),\\ \alpha_{0,0}^{2,0}(t)&=& a_2 a_1 B_t(a_2,a_1,a_0)=\frac{a_1}{a_2-a_1}\exp{(-a_2 t)}-\frac{a_2}{a_2-a_1}\exp{(-a_1 t)}+1,
\end{eqnarray*}
with
 $a_2=2(2+\delta'+\delta)$, $a_1=\delta'+\delta, a_0=0$. And so on $\ldots.$

 Note also that Proposition \ref{mixiji} and result
  (\ref{final2}) can be explained by spectral properties of the
  transition operator $P_t$. Indeed, considered as an operator on the
  space $L^{2}(\pi(x)dx)$, it has a sequence of eigenvalues and an
  orthonormal basis of eigenfunctions. The eigenvalues are exactly the
  $(\exp{(-a_{n} t)}, n \ge 0)$. The eigenfunction associated with
  $\exp{(-a_n t)}$ is a polynomial of degree $n$, linked with the $n$-th Jacobi polynomial  with indexes
  $(\frac{\delta'}{2}-1,\frac{\delta}{2}-1)$ (see the Appendix). Thus, each polynomial
  $h_{i,j}$ has a finite  expansion on this eigenfunctions
  basis. Therefore, $P_t h_{i,j}$ has also a finite expansion on the
  same basis. However, from these spectral properties, it is not
  evident to guess the expansion obtained in (\ref{final2}) nor is it
  to guess that the expansion contains only positive terms that lead
  to mixture coefficients.

Finally, Proposition \ref{mixiji} shows that, for all $t \ge 0$, $\sum_{0\le k \le i,0 \le l \le j-i}\alpha^{i,j-i}_{i-k,j-i-l}(t) =1.$ This can be checked directly by formulae (\ref{alphaiji})-(\ref{alpha}) (see the Appendix).

\subsection{Working the filtering-prediction algorithm and  estimating
  unknown parameters.}
We must now illustrate how Proposition \ref{suffi} allows to obtain
explicitly the successive distributions of filtering $\nu_{n|n:1}$ and
of (one-step) prediction $\nu_{n+1|n:1}$ (see (\ref{cond})). Suppose
that the initial distribution is ${\cal L}(X_1)= \nu_{0,0}$, {\em i.e.} the stationary distribution of $(x(t))$. After one observation $Y_1$, we have the up-dated distribution $\nu_{1|1:1}=\varphi_{Y_1}(\nu_{0,0})= \nu_{Y_1, 1-Y_1}$. Then, we apply the prediction step to get $\nu_{2|1:1}= \nu_{Y_1, 1-Y_1} P_{\Delta}$. This distribution is obtained by Proposition \ref{mixiji}:
\begin{equation}
\nu_{Y_1, 1-Y_1} P_{\Delta}= \sum_{0 \le k \le Y_1, 0 \le l \le 1-Y_1} \alpha^{Y_1, 1-Y_1}_{Y_1 - k,1-Y_1 -l}(\Delta) \nu_{Y_1 -k,1-Y_1 -l}.
\end{equation}
Then, there is another up-dating for $Y_2$, and another prediction, and so on. 
To be more precise, let us state a proposition that explains the use of Proposition \ref{suffi}.
\begin{prop} \label{algo} Suppose $\nu= \sum_{0 \le k \le i, 0 \le l \le j} \alpha_{i-k,j-l}\; \nu_{i-k,j-l}$ is a distribution of ${\bar {\cal F}}_f$.
\begin{enumerate}
\item  Then, for $y=0,1$, 
\begin{equation}\nonumber
\varphi_y(\nu) \propto \sum_{0 \le k \le i, 0 \le l \le j}
\alpha_{i-k,j-l} \; p_{\nu_{i-k,j-l}}(y)\; \nu_{i+y-k,j+1-y-l},
\end{equation}
where the marginal distribution $p_{\nu_{i-k,j-l}}(y)$ is given in (\ref{margij}). Thus, 
\begin{equation}\nonumber
\varphi_y(\nu) =\sum_{0 \le k \le i+y, 0 \le l \le j+1-y} {\hat \alpha_{i+y-k,j+1-y-l}} \;\nu_{i+y-k,j+1-y-l},
\end{equation}
where ${\hat \alpha_{i+y-k,j+1-y-l}}\propto \alpha_{i-k,j-l}\; p_{\nu_{i-k,j-l}}(y)$ for $k=0, 1,\ldots,i,l=0, 1, \ldots,j$ and ${\hat \alpha_{i+y-k,j+1-y-l}}=0$ otherwise.
\item We have:
\begin{equation}\nonumber
\nu P_{\Delta}=\sum_{0 \le \kappa \le i, 0 \le \lambda \le j}\left(\sum_{0 \le k \le \kappa,0 \le l \le \lambda} \alpha_{i-k,j-l}\;\alpha^{i-k,j-l}_{i-\kappa,j-\lambda}(\Delta) \right) \nu_{i-\kappa,j-\lambda},
\end{equation}
where the $\alpha^{i-k,j-l}_{i-\kappa,j-\lambda}(\Delta)$ are given in
Proposition \ref{mixiji}. 
\item The marginal distribution associated with $\nu$ is 
\begin{equation}\label{margmel}
p_{\nu}(y)= \sum_{0 \le k \le i, 0 \le l \le j} \alpha_{i-k,j-l}\; p_{\nu_{i-k,j-l}}(y).
\end{equation}
It is therefore a mixture of Bernoulli distribution (see Proposition \ref{c1} and formula (\ref{margij})).
\end{enumerate}
\end{prop}
The first part is straightforward. The second part is an application
of Proposition \ref{mixiji} with an interchange of sums. Thus, the
number of components in the successive mixture distributions
grows. Indeed, let us compute the number of mixture components for the
filtering distributions. For $\nu_{1|1:1}$, we find
$(1+Y_1)(1+1-Y_1)$; the prediction step preserves the number of
components. For $\nu_{n|n:1}$ and $\nu_{n+1|n:1}$, the number of
components is $(1+\sum_{i=1}^{n}Y_i)(1+n-\sum_{i=1}^{n}Y_i)$. However, as noted above, very few mixture coefficients will be significantly non nul. It was also the case for the model investigated in Genon-Catalot and Kessler (2004). 

Let us notice that the $h$-step ahead predictive distribution, $\nu_{n+h|n:1}$ is obtained from $\nu_{n|n:1}$ by applying the operator $P_{\Delta}^{h}$, {\em i.e.} $\nu_{n+h|n:1}=\nu_{n|n:1}P_{\Delta}^{h}$. Therefore, this distribution stays in the class ${\bar {\cal F}}_{f}$ and has the same number of mixture components as $\nu_{n|n:1}$.

Now, suppose that $\delta', \delta$ are unknown and that we wish to
estimate these parameters using the data set $(Y_1, \ldots, Y_n)$. The
classical statistical approach is to compute the corresponding
maximum likelihood estimators. This requires the computation of the
exact joint density of this data set which gives the likelihood
function (see (\ref{likelihood})). For general hidden Markov
models, the exact formula of this density is difficult to handle since the integrals
giving the  conditional
densities of $Y_i$ given $(Y_{i-1}, \ldots, Y_1)$  are not explicitly
computable (see formula (\ref{marginalen})). On the contrary, in our
model, these integrals are computable by formula
(\ref{margij}). Suppose that the initial distribution is the
stationary distribution of (\ref{wf}), {\em i.e.} $\nu_{0,0}$. For
$i=1$, the law of $Y_1$ has density $p_{\nu_{0,0}}(y_1)$ given by
(\ref{margij}): It is a Bernoulli distribution with parameter
$\frac{\delta'}{\delta'+\delta}$. Then, for $i \ge 2$, the conditional
distribution of $Y_i$ given $Y_{i-1}, \ldots, Y_1$ has density
$p_{\nu_{i|i-1:1}}(y_i)$. It is now a mixture of Bernoulli
distributions. The exact likelihood is therefore a product of mixtures
of Bernoulli distributions. 
\subsection{Marginal smoothing.}\label{margsmooth}
In this section, we compute $\nu_{l|n:1}$ for $l<n$. To simplify notations, denote by $p(x_l|y_n, \ldots, y_1)$ the
conditional density of $X_l$ given $Y_n=y_n, \ldots, Y_1=y_1$, {\em
  i.e.} the density of $\nu_{l|n:1}$ taken at $Y_n=y_n, \ldots,
Y_1=y_1$. Analogously, denote by $p(y_i|y_{i-1}, \ldots, y_1)$ the
conditional density of $Y_i$ given $Y_{i-1}=y_{i-1}, \ldots,
Y_1=y_1$. We  introduce the backward function:
\begin{equation} \label{vraiscond}
p_{l,n}(y_{l+1}, \ldots,y_n;x), 
\end{equation}
equal to the conditional density of $(Y_{l+1}, \ldots, Y_n)$ given $X_l=x$. By
convention, we set $ p_{n,n}(\emptyset;x)=1$. Then, the following
forward-backward decomposition holds.
\begin{prop} For $l\le n$,
\begin{equation} \label{smooth}
p(x_l|y_n, \ldots, y_1)= \frac{p(x_l|y_l, \ldots,
  y_1)}{\prod_{i=l+1}^{n} p(y_i|y_{i-1}, \ldots, y_1)} p_{l,n}(y_{l+1}, \ldots,y_n;x_l)
\end{equation}
\end{prop}
This result is classical and may be found {\em e.g.} in Capp\'e {\em
  et al.} (2005). Therefore, the smoothing density is obtained
using the filtering density that we have already computed. The denominator in
(\ref{smooth}) is also available. It remains to have a more explicit
expression for the backward function (\ref{vraiscond}). 
The following proposition gives a backward recursion from $l=n-1$ down
to $l=1$ for computing (\ref{vraiscond}).
\begin{prop} First, for all $n$,
\begin{equation} \label{n-1n}
p_{n-1,n}(y_{n};x)= P_{\Delta}[f_{.}(y_{n})](x).
\end{equation}
Then, for $l+1 < n $,
\begin{equation} \label{ln}
p_{l,n}(y_{l+1}, \ldots,y_n;x)= P_{\Delta}[f_{.}(y_{l+1})
p_{l+1,n}(y_{l+2}, \ldots,y_n;.)](x)
\end{equation}
\end{prop}
\noindent
\begin{proof} We use the fact that $(X_n,Y_n)$ is Markov with
transition $p_{\Delta}(x_n,x_{n+1}) f_{x_{n+1}}(y_{n+1})$. Given
$X_{n-1}=x$, $X_{n}$ has distribution
$p_{\Delta}(x,x_{n})dx_{n}$. Hence, 
\begin{equation}\nonumber
p_{n-1,n}(y_{n};x)= \int_0^{1} p_{\Delta}(x,x_{n}) f_{x_{n}}(y_{n})dx_{n},
\end{equation}
which gives (\ref{n-1n}). Then, for $n \ge l+2$, 
\begin{equation} \nonumber
p_{l,n}(y_{l+1}, \ldots,y_n;x)
\end{equation}
\begin{equation} \nonumber
= \int_0^{1}p_{\Delta}(x,x_{l+1})
f_{x_{l+1}}(y_{l+1}) \times \prod_{i=l+2}^{n}
p_{\Delta}(x_{i-1},x_{i}) f_{x_{i}}(y_{i}) dx_{l+1}\ldots dx_{n}
\end{equation}
\begin{equation} \nonumber
= \int_0^{1}p_{\Delta}(x,x_{l+1})
f_{x_{l+1}}(y_{l+1}) p_{l+1,n}(y_{l+2}, \ldots,y_n;x_{l+1})dx_{l+1},
\end{equation}
which gives (\ref{ln}).
\end{proof}
Let us now apply these formulae to our model.  We will show briefly that backward functions can be computed by simple application of Theorem \ref{miji}. Indeed,  since 
\begin{equation}\nonumber
f_x(y_n)= h_{y_n,1-y_n}(x),
\end{equation}
\begin{equation}\nonumber
p_{n-1,n}(y_n;x)= P_{\Delta}h_{y_n,1-y_n}(x)= m_{y_n,1-y_n}(\Delta,x),
\end{equation}
is obtained by Theorem \ref{miji}. Next, we compute
\begin{equation}\nonumber
m_{y_n,1-y_n}(\Delta,.) \times h_{y_{n-1},1-y_{n-1}}(.),
\end{equation}
which is a linear combination of $h_{y_{n-1}+y_n-k,2-y_{n-1}-y_n-l}$ with $0 \le k \le y_n, 0 \le l \le 1-y_n$ and apply the transition operator $P_{\Delta}$ to  get $p_{n-2,n}(y_{n-1},y_n;x)$. This is again given by Theorem \ref{miji}. By elementary induction, we see that backward functions are explicit.

\noindent

\section{Appendix}\label{appendix}

\subsection{Proof of Lemma \ref{lemA}.}
Let us write in more details expression  (\ref{formule4}). We have
\begin{eqnarray*}
A&=&\frac{1}{(a_n-a_{n-1})}\times
\frac{1}{(a_{n-1}-a_{n-2})(a_{n-1}-a_{n-3})\ldots(a_{n-1}-a_{n-1-k})}\\
&&+\frac{(-1)}{(a_n-a_{n-2})(a_{n-1}-a_{n-2})}\times
\frac{1}{(a_{n-2}-a_{n-3})\ldots\ldots(a_{n-2}-a_{n-1-k})}+\ldots\\
&&+
\frac{(-1)^{j-1}}{(a_n-a_{n-j})(a_{n-1}-a_{n-j})\ldots(a_{n-j+1}-a_{n-j})}\times
\frac{1}{(a_{n-j}-a_{n-j-1}) \ldots(a_{n-j}-a_{n-k-1})}\\
&& +\ldots + \frac{(-1)^{k}}{(a_n-a_{n-1-k})(a_{n-1}-a_{n-1-k})(a_{n-2}-a_{n-1-k})\ldots(a_{n-k}-a_{n-1-k})}
 \end{eqnarray*}
Now, we set
\begin{eqnarray*}
L_0&=&(a_n - a_{n-1})(a_n - a_{n-2})(a_n-a_{n-3})\ldots (a_{n}-a_{n-k-1}),\\
L_1&=&\hspace*{.4cm} (a_{n-1}-a_{n-2})(a_{n-1}-a_{n-3})\ldots\ldots (a_{n-1}-a_{n-k-1}),\\
L_2&=&\hspace*{.8cm}(a_{n-2}-a_{n-3})(a_{n-2}-a_{n-4}) \ldots
(a_{n-2}-a_{n-k-1}),\\\vdots&&\vdots\\
L_{k-1}&=&\hspace*{3.2cm}(a_{n-k+1}-a_{n-k})(a_{n-k+1}-a_{n-k-1}),\\ 
L_k&=&\hspace*{6cm}(a_{n-k}-a_{n-k-1}).
\end{eqnarray*}
We must prove that 
\begin{equation}\label{AL0}
A=\frac{1}{L_0}.
\end{equation}
 For this, we introduce the
product $T_n=L_0 L_1\ldots L_k$.  Now, we need to prove that $A T_n=
L_1 L_2\ldots L_k$. We start to compute $A T_n$:
\begin{eqnarray*}
A T_n &=& \frac{L_0}{a_n-a_{n-1}} L_2 \ldots L_k \\
&&+ \frac{(-1) L_0 L_1}{(a_n-a_{n-2})((a_{n-1}-a_{n-2})}L_3 \ldots
L_k+\ldots\\
&&+\frac{(-1)^{j-1}L_0L_1\ldots
  L_{j-1}}{(a_n-a_{n-j})(a_{n-1}-a_{n-j})\ldots
  (a_{n-j+1}-a_{n-j})}L_{j+1}\ldots L_k +\ldots\\
&&+\frac{(-1)^{k}L_0 L_1\ldots L_{k-1}}{(a_n-a_{n-k-1})(a_{n-1}-a_{n-k-1})\ldots(a_{n-k+1}-a_{n-k-1})}.
\end{eqnarray*}
Now, we see that
\begin{equation}\label{p}
A T_n= P(a_n)
\end{equation}
 where $P(.)$ is a polynomial with degree $k$. Indeed, in $A T_n$, the
 terms containing $a_n$ come only from the terms
 \begin{equation} \nonumber
\frac{L_0}{(a_n-a_{n-j})}=P_j(a_n),
\end{equation}
where
\begin{equation}\nonumber
P_1(x)=(x-a_{n-2})(x-a_{n-3})\ldots(x-a_{n-k-1}),
\end{equation}
\begin{equation}\nonumber
P_j(x)= (x-a_{n-1})(x-a_{n-2})\ldots(x-a_{n-j+1})\times(x-a_{n-j-1})\ldots(x-a_{n-k-1}),
\end{equation}
\begin{equation}\nonumber
P_{k+1}(x)= (x-a_{n-1})(x-a_{n-2})\ldots(x-a_{n-k}),
\end{equation}
are all products of $k$ factors of degree $1$. Notice that $P_j(x)$ is
nul for $x=a_{n-1}, a_{n-2},  \ldots, a_{n-j+1},a_{n-j-1},\ldots,
a_{n-k-1}$ and that $P_1(a_{n-1})=L_1$, 
\begin{equation}\nonumber
P_j(a_{n-j})= (a_{n-j}-a_{n-1})(a_{n-j}-a_{n-2})\ldots(a_{n-j}-a_{n-j+1})\times L_j,
\end{equation}
\begin{equation}\nonumber
P_{k+1}(a_{n-k-1})= (a_{n-k-1}-a_{n-1})(a_{n-k-1}-a_{n-2})\ldots (a_{n-k-1}-a_{n-k}).
\end{equation}
Therefore (see (\ref{p}))
\begin{eqnarray*}
P(x)&=& P_1(x) L_2L_3\ldots L_k\\
&& -P_2(x) \frac{L_1}{a_{n-1}-a_{n-2}}L_3\ldots L_k +\ldots\\
&& + (-1)^{j-1}P_j(x) \frac{L_1 L_2\ldots L_{j-1}}{(a_{n-1}-a_{n-j})(a_{n-2}-a_{n-j})\ldots
  (a_{n-j+1}-a_{n-j})}L_{j+1}\ldots L_k +\ldots\\
&&+(-1)^{k} P_{k+1}(x) \frac{L_1\ldots
  L_{k-1}}{(a_{n-1}-a_{n-1-k})(a_{n-2} -a_{n-1-k})\ldots
  (a_{n-k+1}-a_{n-1-k})}
\end{eqnarray*}
Now,
\begin{eqnarray*}
P(a_{n-1})&=&P_1(a_{n-1})L_2L_3\ldots L_k= L_1L_2L_3\ldots L_k\\
P(a_{n-2})&=&-P_2(a_{n-2}) \frac{L_1}{(a_{n-1}-a_{n-2})}L_3\ldots L_k=
-\frac{(a_{n-2}-a_{n-1})}{(a_{n-1}-a_{n-2})}L_1L_2\ldots L_k\\
&=&L_1L_2\ldots L_k\\
\vdots &&\vdots \\
P(a_{n-j})&=& (-1]^{j-1} P_j(a_{n-j})\frac{L_1 L_2\ldots L_{j-1}}{(a_{n-1}-a_{n-j})(a_{n-2}-a_{n-j})\ldots
  (a_{n-j+1}-a_{n-j})}L_{j+1}\ldots L_k\\
&=&(-1)^{2(j-1)}L_1L_2\ldots L_{j-1} L_j L_{j+1}\ldots L_k \\
\vdots &&\vdots\\
P(a_{n-k-1})&=& (-1)^{k} P_{k+1}(a_{n-k-1})) \frac{L_1\ldots
  L_{k-1}}{(a_{n-1}-a_{n-1-k})(a_{n-2} -a_{n-1-k})\ldots
  (a_{n-k+1}-a_{n-1-k})}\\
&=&(-1)^{2k} L_1 L_2\ldots L_k.
\end{eqnarray*}
Therefore, $P(x)=L_1L_2\ldots L_k$ for the $k+1$ distinct values
$x=a_{n-1},a_{n-2}, \ldots,a_{n-k-1}$. Since $P(x)$ is a polynomial of
degree $k$, it is constant equal to $L_1L_2\ldots L_k$. In particular,
$$
P(a_n)=L_1L_2 \ldots L_k,
$$
which is equivalent to $A= 1/L_0$ (see (\ref{AL0})). So the proof of
Lemma \ref{lemA} is complete.

\subsection{Binomial or negative binomial conditional distributions.}
Proposition \ref{c1} holds for the other cases given in the introduction. Let $\nu_{i,j}$ belong to ${\cal F}$.
\begin{itemize}
\item If $f_x(y)= \binom{N}{y}x^y (1-x)^{N-y}$, $y=0, \ldots, N$, then $\varphi_y(\nu_{i,j})= \nu_{i+y, j+N-y}$ and the marginal distribution is equal to:
\begin{equation}
p_{\nu_{i,j}}(y)=
\binom{N}{y}\frac{B(i+y+(\delta'/2),j+N-y+(\delta/2))}{B(i+(\delta'/2),j+(\delta/2))},
y=0,1, \ldots,N 
\end{equation}
\item  If $f_x(y)=\binom{m+y-1}{y} x^m (1-x)^y$, $y=0, \ldots$, then $\varphi_y(\nu_{i,j})= \nu_{i+m, j+y}$ and for $y=0,1, \ldots,$
\begin{equation}
p_{\nu_{i,j}}(y)=
\binom{m+y-1}{y}\frac{B(i+m+(\delta'/2),j+y+(\delta/2))}{B(i+(\delta'/2),j+(\delta/2))}\\
\end{equation}
\end{itemize}
\subsection{Mixture coefficients.}
We now check using formula (\ref{alphaiji}) that $\sum_{0\le k \le i, 0 \le l \le j-i}\alpha^{i,j-i}_{i-k,j-i-l}(t)=1.$ By  interchanging sums and setting $k'=k, l'=k+l$, we first get
\begin{equation}\nonumber
\sum_{0\le k \le i, 0 \le l \le j-i}\alpha^{i,j-i}_{i-k,j-i-l}(t)= \sum_{l'=0}^{j}p(i,j-i) a_ja_{j-1}\ldots a_{j-l'+1}B_t(a_j,\ldots, a_{j-l'}),
\end{equation}
where
\begin{equation}\nonumber
p(i,j-i)=\sum_{0\le k'\le i, 0 \le l'-k'\le j-i} \frac{ \binom{i}{k'} \binom{j-i}{l'-k'}}{\binom{j}{l'}}.
\end{equation}
We recognize the sum of hypergeometric probabilities so that $p(i,j-i)=1$. There remains to prove that, for all $i\ge 0$,
\begin{equation}\nonumber
\sum_{0\le k \le i}\alpha^{i,0}_{i-k,0}(t)= \sum_{k=0}^{i} a_i a_{i-1}\ldots a_{i-k+1}B_t(a_i,\ldots, a_{i-k})=1.
\end{equation}
We fix $i$. Looking at (\ref{btank}) and  interchanging sums, we have to check
that
\begin{equation}\nonumber
\sum_{j=0}^{i} H_{i-j} \exp{(-a_{i-j} t)}=1,
\end{equation}
where, for $j=0,1, \ldots,i$, 
\begin{equation}\label{Hij}
H_{i-j}=  \sum_{k=j}^{i} L_k^{j}
\end{equation}
and 
\begin{equation}\nonumber
L_k^{j}=(-1)^{k+j} \frac{ a_i a_{i-1}\ldots
  a_{i-k+1}}{\prod_{0\le l \le k, l\neq j}|a_{i-j} - a_{i-l}|}.
\end{equation}
Since $a_0 =0$ and $H_0= (-1)^{2j} a_i \ldots a_1/ a_i \ldots a_1= 1$,
we have $H_0 \exp{(-a_0 t)}= 1$. So we must prove that, for all $j=0,
1, \ldots, i-1$, $H_{i-j}=0$. Denote by $D_k^{j}$ the denominator of
$L_k^{j}$:
$$
D_k^{j}=(a_i - a_{i-j}) \ldots (a_{i-j+1} - a_{i-j})(a_{i-j}-
a_{i-j-1})\ldots (a_{i-j}-a_{i-k}).
$$
It is easy to prove by induction (on $k$) that, for $k=i-1, \ldots, j+1$,
\begin{equation}\nonumber
{L'}_k^{j}:=L_i^{j} + L_{i-1}^{j} +\ldots + L_k^{j}= (-1)^{k+j}
\frac{a_i\ldots {\hat a_{i-j}} \ldots a_{i-k+1}}{D_{k-1}^{j}}, 
\end{equation}
where the notation ${\hat .}$ means that the term is absent. The
formula for $k=j+1$ yields 
\begin{equation}\nonumber
{L'}_{j+1}^{j}= (-1)^{2j+1} \frac{a_i \ldots a_{i-j+1}}{D_j^{j}}= - L_j^{j}.
\end{equation}
This gives $H_{i-j}=0$ (see (\ref{Hij})) for all $j=0,1, \ldots, i-1$..
\subsection{Spectral approach.}
The transition density $p_t(x,y)$ of (\ref{wf}) can be expressed using
the spectral decomposition of the operator $P_t$. 
Consider equation (\ref{wf}) and set $z(t)= 2 x(t)-1$. Then, 
\begin{equation}\nonumber
dz(t)= [-\delta (1+z(t))+ \delta'(1-z(t)]dt + (1-z^{2}(t))^{1/2} dW_t.
\end{equation}
This is a Jacobi diffusion process. Let us set 
\begin{equation} \label{alphabeta}
\alpha= \frac{\delta}{2}-1, \beta= \frac{\delta'}{2}-1.
\end{equation}
Then, for $n \ge 0$, $u(z)= P_n^{\alpha, \beta}(z)$  with 
\begin{equation}\label{jacobi}
P_n^{\alpha, \beta}(z)= \frac{(-1)^{n}}{2^{n} n!}
(1-z)^{-\alpha}(1+z)^{-\beta} \frac{d^{n}}{dz^{n}}[(1-z)^{n+\alpha} (1+z)^{n+\beta}],
\end{equation}
is solution of 
\begin{equation} \label{ode}
(1-z^{2})u'' + [\beta - \alpha -(\alpha+\beta +2)z]u'= -n(n+\alpha+\beta+1)u.
\end{equation}
The function (\ref{jacobi}) is the Jacobi polynomial of degree
$n$ with indexes $(\alpha,\beta)$. The sequence  $(P_n^{\alpha, \beta}(z), n \ge 0)$ is an 
orthogonal family with respect to the weight function
$\rho(z)=(1-z)^{\alpha}(1+z)^{\beta}1_{(-1,+1)}(z)$. After normalization, it
constitutes an orthonormal basis of $L^{2}(\rho(z)dz)$ (see {\em e.g.}
Lebedev (1972, p.96-97) or Nikiforov and Ouvarov (1983, p.37)). Now,
we set $h(x)=u(2x-1)$ in (\ref{ode}) and get:
\begin{equation}\nonumber
2x(1-x)h'' +[\beta - \alpha -(\alpha+\beta +2)(2x-1)]h' = - 2 n(n+\alpha+\beta+1)h.
\end{equation}
Using the relations (\ref{alphabeta}), we obtain:
\begin{equation}\nonumber
2x(1-x)h'' +[-\delta x + \delta'(1-x)]h' = -
n(2(n-1)+\delta+\delta')h.
\end{equation}
Hence, $Lh== -a_n h$ where $L$ is the infinitesimal generator of (\ref{wf}).  For $n \ge 0$, the sequence
\begin{equation}\nonumber
Q_n(x)= P_n^{ \frac{\delta}{2}-1, \frac{\delta'}{2}-1}(2x-1)
\end{equation}
is the sequence of eigenfunctions of  $L$. The eigenvalue associated with $Q_n$ is $-a_n$. The transition operator $P_t$ has the same sequence of eigenfunctions,
and the eigenvalues are $(\exp{(-a_nt)})$. We have:
\begin{equation}\nonumber
Q_n(x)= \frac{(-1)^{n}}{ n!}
x^{-( \frac{\delta'}{2}-1)}(1-x)^{-(\frac{\delta}{2}-1)} \frac{d^{n}}{dx^{n}}[(1-x)^{n+ \frac{\delta'}{2}-1} x^{n+\frac{\delta}{2}-1}].
\end{equation}
Each polynomial $Q_n$ is of the form (see (\ref{hij}))
\begin{equation}\label{qn}
Q_n(x)= \sum _{i=0}^{n} c_{i,n-i}^{n} h_{i,n-i}(x).
\end{equation}
And each $h_{i,j-i}$ can be developped as
\begin{equation}\nonumber
h_{i,j-i}= \sum_{k=0}^{j} d_{k}^{i,j-i} Q_k,
\end{equation}
with $d_k^{i,j-i}=c_k^{-1/2} \int_0^{1} h_{i,j-i}(x)Q_k(x)\pi(x)dx $
and $c_k= \int_0^{1}Q_k^{2}(x)\pi(x)dx$. Since $P_t Q_k= \exp{(-a_k t)} Q_k$, 
\begin{equation}\nonumber
P_t h_{i,j-i}= \sum_{k=0}^{j} \exp{(-a_k t)}d_{k}^{i,j-i} Q_k.
\end{equation}
This approach requires the computation of the coordinates $d_k^{i,j-i}$ and of the coefficients $c_{i,k-i}^{k}$ of (\ref{qn}). Our method gives directly the expression of $P_t h_{i,j-i}$.

Let us notice that the transition density of (\ref{wf}) has the following expression:
\begin{equation}\label{transition}
p_t (x,y)= \pi(y) \sum_{n=0}^{+\infty} \exp{(-a_n t)} Q_n(x)Q_n(y) c_n^{-1},
\end{equation}
as explained in Karlin and Taylor (1981).  Therefore, by using the
expression (\ref{qn}) and some computations, it is possible to prove
that this transition satisfies also condition (T1) of Chaleyat-Maurel
and Genon-Catalot (2006). More precisely, this transition can be
expressed as an infinite mixture of distributions of the class ${\cal
  F}$.  

This property has the following consequence. Suppose that the initial
variable in (\ref{wf}) is deterministic $x(0)=x_0$. Then, $x(t_1)$ has
distribution $p_{t_1}(x_0, x)$. This distribution belongs to the
extended class ${\bar {\cal F}}$ composed of infinite mixtures of
distributions of ${\cal F}$. We can apply our results to the extended
class: The filtering, prediction or smoothing distributions all belong
to ${\bar {\cal F}}$. 
\end{document}